\documentclass[12 pt,reqno]{amsart}
\usepackage{amssymb,amsmath,amstext,amsgen,amsfonts,amsthm,verbatim,hyperref,fullpage,enumerate}
\usepackage[usenames,dvipsnames]{color}
\usepackage{amsrefs}

\newcommand{\R}{\mathbb{R}}

\newcommand{\Z}{\mathbb{Z}}

\newcommand{\Q}{\mathbb{Q}}
\renewcommand{\H}{\mathbb{H}}

\newcommand{\g}{\mathfrak{g}}
\newcommand{\h}{\mathfrak{h}}

\renewcommand{\epsilon}{\varepsilon}
\renewcommand{\phi}{\varphi}
\renewcommand{\hat}{\widehat}

\newcommand{\cB}{\mathcal{B}}

\newcommand{\cF}{\mathcal{F}}

\newcommand{\Hd}{\mathcal{H}}
\newcommand{\cL}{\mathcal{L}}

\newcommand{\cV}{\mathcal{V}}
\newcommand{\dist}{\mathrm{dist}}
\DeclareMathOperator{\diam}{diam}

\newtheorem{theorem}{Theorem}[section]
\newtheorem{proposition}[theorem]{Proposition}
\newtheorem{lemma}[theorem]{Lemma}

\theoremstyle{remark}
\newtheorem{remark}[theorem]{Remark}

\begin{document}
\title{BiLipschitz decomposition of Lipschitz maps between Carnot groups}

\author{Sean Li}
\date{\today}
\address{Department of Mathematics, The University of Chicago, Chicago, IL 60637}
\email{seanli@math.uchicago.edu}

\begin{abstract}
  Let $f : G \to H$ be a Lipschitz map between two Carnot groups.  We show that if $B$ is a ball of $G$, then there exists a subset $Z \subset B$, whose image in $H$ under $f$ has small Hausdorff content, such that $B \backslash Z$ can be decomposed into a controlled number of pieces, the restriction of $f$ on each of which is quantitatively biLipschitz.  This extends a result of \cite{meyerson}, which proved the same result, but with the restriction that $G$ has an appropriate discretization.  We provide an example of a Carnot group not admitting such a discretization.
\end{abstract}

\maketitle

\section{Introduction}
Let $\Hd^n_\infty$ denote the Hausdorff $n$-content (the definition will be reviewed in the next section).  We prove the following theorem.
\begin{theorem} \label{t:main}
  Let $G$ and $H$ be Carnot groups endowed with their Carnot-Carath\'eodory metric and $N$ be the Hausdorff dimension of $G$.  There exists some $c > 0$ depending only on $G$ and $H$ with the following property.  For each $\delta > 0$, there exists some $M = M(\delta) > 1$ so that if $x \in G$, $R > 0$, and $f : B(x,R) \to H$ is a 1-Lipschitz function, then there exist closed sets $\{F_i\}_{i=1}^M$ of $B(x,R)$ so that
  \begin{align*}
    \Hd^N_\infty \left( f\left( B(x,R) \backslash \bigcup_{i=1}^M F_i \right) \right) < c\delta R^N,
  \end{align*}
  and
  \begin{align*}
    |f(x) - f(y)| \geq \delta |x-y|, \qquad \forall x,y \in F_i, \forall i \in \{1,...,M\}.
  \end{align*}
\end{theorem}

The first results of these type were published independently by P. Jones \cite{jones} and G. David \cite{david} (in the same issue of the same journal) and were motivated by problems in the field of singular integrals.  Jones proved Theorem \ref{t:main} for Lipschitz maps $f : [0,1]^n \to \R^m$.  Jones's result was later generalized by Schul in \cite{schul} where he showed the same result except now $f : [0,1]^n \to (X,d)$ takes value in a general metric space and can be a Lipschitz map up to a controlled additive error.  The result was also generalized by G.C. David (not the same author of \cite{david}) in \cite{david-gc} where he proved the result for maps between certain topological manifolds of equivalent (topological and Hausdorff) dimensions.

Carnot groups are a class of metric groups that are natural generalizations of Euclidean spaces.  They have many familiar geometric qualities, including properness, geodicity, a dilation structure, and transitiveness of isometries (in fact, these four attributes characterize sub-Finsler Carnot group \cite{ledonne}) although they may not be abelian.  Carnot groups also admit a class of nested dyadic cubes (to be described in the next section) that allows one to import many arguments from harmonic analysis.  Thus, it is natural to ask whether certain analytic or geometric statements can be generalized from the Euclidean world to the Carnot world.  This has been an active area of research.

BiLipschitz decomposition for Lipschitz maps between Carnot groups was first studied in \cite{meyerson}.  There, the author proved Theorem \ref{t:main} when $H$ is a Carnot group and $G$ is another Carnot group that is appropriately discretizable (see Definition 2.11 of the same paper).  There are examples of Carnot groups that cannot be discretized, and we give an example in the next section.  Thus, our theorem extends the result of \cite{meyerson} to arbitrary pairs of Carnot groups.  It should be noted that the result of \cite{david-gc} gives that self-maps of any Carnot group satisfy Theorem \ref{t:main} as Carnot groups satisfy the topological manifold conditions studied there.  However, the result of \cite{david-gc} does not apply for maps between Carnot groups of different dimensions.


Let $G$ be a $N$-dimensional Carnot group and $X$ be an arbitrary metric space.  One natural question now is, if $f: G \to X$ is a Lipschitz map with positive Hausdorff $N$-measure image there, does there always exists a positive measure subset of $G$ on which $f$ is biLipschitz?  For the specific case when $G = \H$, the Heisenberg group (which has Hausdorff dimension 4) this was asked as Question 24 in \cite{heinonen-semmes}.  In a future paper joint with E. Le Donne and T. Rajala, we show this is not possible.  Recall that a metric space is Ahlfors $n$-regular if there exists some $C \geq 1$ so that
$$\frac{1}{C} r^n \leq \Hd^n(B(x,r)) \leq Cr^n, \qquad \forall x \in X, r < \diam X.$$
We construct in \cite{llr} an Ahlfors 4-regular metric space that the Heisenberg group Lipschitz surjects onto, but which has no biLipschitz pieces.  Thus, one cannot hope for an analogue of Theorem \ref{t:main} when the target space is an arbitrary metric space.

BiLipschitz decomposition theorems all follow a similar strategy.  The first step is to decompose the domain into a set of dyadic-like cubes.  Then one proves a lemma stating that if a cube whose image under $f$ has large Hausdorff content and a certain wavelet coefficient-like quantity of $f$ on the cube is small, then $f$ acts biLipschitzly on points of the cube that are far apart.  This is the step that usually heavily involves the geometry of the setting.  Lemma \ref{l:weak-bilip} gives this statement for us.  One then uses the fact that a weighted sum of the wavelet coefficient-like quantity is bounded to show that the quantity cannot be big for many cubes.  After throwing out the cubes that have small image, one can then decompose most of the rest of the domain into a controlled number of pieces on which $f$ is biLipschitz using a coding scheme.  For us, the wavelet coefficient-like quantity will be the deviation of $f$ from an affine function on a cube, a quantity that was studied in \cite{li}. \\

\noindent {\bf Acknowledgements.} I am grateful to Jonas Azzam, Enrico Le Donne, and Pierre Pansu for enlightening discussions.  The research presented here is supported by NSF postdoctoral fellowship DMS-1303910.

\section{Preliminaries and notations}
Given a metric space $(X,d)$, a subset $E \subseteq X$, and two numbers $N \geq 0$ and $\delta \in (0,\infty]$, we define
\begin{align*}
  \Hd^N_\delta(E) := \inf \left\{ \sum_{i=1}^\infty (\diam A_i)^N : E \subseteq \bigcup_{i=1}^\infty A_i, ~\diam A_i \leq \delta \right\}.
\end{align*}
Note that if $s \leq t$, then $\Hd^N_s(E) \geq \Hd^N_t(E)$.  We then let $\Hd^N(E) := \lim_{\delta \to 0} \Hd^N_\delta(E)$ be the Hausdorff $N$-measure and $\Hd^N_\infty(E)$ be the Hausdorff $N$-content.  It follows that $\Hd^N_\infty(E) \leq \Hd^N(E)$ always.  In particular, if $\Hd^N(E) = 0$, then $\Hd^N_\infty(E) = 0$ (the reverse implication also holds).

We will use the convention that if $\lambda > 1$ and $E \subset X$, then
\begin{align}
  \lambda E := \{x \in X : \dist(x,E) \leq (\lambda - 1) \diam E\}. \label{e:lambda-E}
\end{align}

A Carnot group $G$ is a simply connected Lie group whose Lie algebra $\g$ is stratified, that is, it can be decomposed into direct sums of subspaces
\begin{align*}
  \g = \bigoplus_{i=1}^r \cV_i
\end{align*}
where $[\cV_1,\cV_j] = \cV_{j+1}$ for $j \geq 1$.  Here, it is understood that $\cV_k = 0$ for all $k > r$.  The layer $\cV_1$ is called the horizontal layer and if $\cV_r \neq 0$ then $r$ is the (nilpotency) step of $G$.  If we are dealing with multiple graded Lie algebras, say $\g$ and $\h$, we will write $\cV_i(\g)$ and $\cV_j(\h)$ instead to differentiate the layers between the different Lie algebras.  For simplicity, we will suppose all the constants in all Lie bracket structures are 1.  All proofs will go through in the general case and the results will only differ by some factor depending on these constants.

As exponential maps are diffeomorphisms between a Carnot group and its Lie algebra, we can use it to canonically identify elements of the Lie group $G$ to the Lie algebra $\g$.  This shows that a Carnot group is topologically a Euclidean space.
From now on, if we write an element of $G$ as $\exp(g_1 + ... + g_r)$, it is understood that $g_i \in \cV_i(\g)$.  We will write the identity element as 0.  Let $|\cdot|$ denote the standard Euclidean norm on $\g$ (viewed as $\R^n$).  Then we can make sense of $|g_i|$ and $|g_i - h_i|$ and so forth.

Group multiplication in Carnot groups $G$ can be expressed in the Lie algebra level $\g$ using the Baker-Campbell-Hausdorff (BCH) formula:
\begin{align*}
  \log(e^U e^V) = \sum_{k > 0} \frac{(-1)^{k-1}}{k} \underset{\substack{r_i+s_i > 0,\\1 \leq i \leq k}}{\sum} a(r_1,s_1,...,r_k,s_k) (ad U)^{r_1} (ad V)^{s_1} \cdots (ad U)^{r_k} (ad V)^{s_k-1} V.
\end{align*}
Here, $(ad X)Y = [X,Y]$ and $a(r_1,s_1,...,r_k,s_k)$ are constants depending only on the Lie algebra structure of $\g$.  Because we are working in the exponential coordinates of $G$, the BCH formula allows us to compute on the level of the coordinates $G$.  Specifically, given $\exp(g_1 + ... + g_r)$ and $\exp(h_1 + ... + h_r) \in G$, we get that
\begin{align*}
  \exp(g_1 + ... + g_r) \cdot \exp(h_1 + ... + h_r) = \exp(g_1 + h_1 + g_2 + h_2 + P_2 + ... + h_r + g_r + P_r)
\end{align*}
where $P_k$ are polynomials in the coordinates $g_1,...,g_{k-1},h_1,...,h_{k-1}$.

An important property of Carnot groups is that they admit a family of dilation automorphisms.  For each $\lambda > 0$, we can define
\begin{align*}
  \delta_\lambda : G &\to G \\
  \exp(g_1 + g_2 + ... + g_r) &\mapsto \exp(\lambda g_1 +  \lambda^2 g_2 + ... + \lambda^r g_r)
\end{align*}
where we use the exponential coordinates of $G$.

A homogeneous norm on a Carnot group $G$ is a function $\mathcal{N} : G \to [0,\infty)$ such that
\begin{align*}
  \mathcal{N}(g) &= \mathcal{N}(g^{-1}), \\
  \mathcal{N}(\delta_\lambda(g)) &= \lambda \mathcal{N}(g), \\
  \mathcal{N}(g) &= 0 \Leftrightarrow g = 0.
\end{align*}
Homogeneous norms induce left-invariant homogeneous (semi)metrics by the formula $d(x,y) = \mathcal{N}(x^{-1}y)$ and vice versa.  Here, $d$ may not satisfy the triangle inequality, but there does exist some $C \geq 1$ so that
\begin{align*}
  d(x,z) \leq C(d(x,y) + d(y,z)), \qquad \forall x,y,z \in G.
\end{align*}
Any two metrics on $G$ induced by two homogeneous norms are biLipschitz equivalent.

We will define a special group norm as
\begin{align*}
  N_\infty : G &\to [0,\infty) \\
  \exp(g_1 + ... + g_r) &\mapsto \max_{1 \leq k \leq r} \lambda_k |g_k|^{1/k}.
\end{align*}
It was shown in \cite[Lemma II.1]{guivarch} (see also \cite[Lemma 2.5]{breuillard}) that for each Carnot group there exists some set of positive scalars $\{\lambda_k\}_{k=1}^r$ so that $d_\infty$, the associated metric, satisfies the actual triangle inequality.  We will suppose for simplicity that $\lambda_k = 1$ for all $k$.  This will change everything we do by only a constant.

Carnot groups also admit a path metric that we describe now.  We begin by constructing a left-invariant tangent subbundle $\Hd$ of the tangent bundle which is just $\cV_1$ pushed to every point by left translation.  We can similarly endow $\Hd$ with a left-invariant field of inner products.  We then define the Carnot-Carath\'eodory metric between $x$ and $y$ in $G$ to be
\begin{align*}
  d_{cc}(g,h) = \inf \left\{ \int_0^1 |\gamma'(t)|_{\gamma(t)} ~dt : \gamma(0) = g, \gamma(1) = h, \gamma' \in \Hd_{\gamma(t)} \right\}.
\end{align*}
Here, $|\cdot|_{\gamma(t)}$ is the norm coming from the left invariant inner product.  For Carnot groups, such a path between any two points always exists (see {\it e.g.} \cite[Chapter 2]{montgomery}) and so the metric is finite.  It is clearly left invariant and scales with dilation by construction.  Thus, the Carnot-Carath\'eodory metric is biLipschitz equivalent to any metric induced by a homogeneous norm.

Let $L : G \to H$ be a Lie group homomorphism between Carnot groups.  As $G$ and $H$ are simply connected, we can then lift it to a linear transform of the Lie algebras $T_L : \g \to \h$ by the formula $T_L = \exp^{-1} \circ L \circ \exp$.  We will always assume that $T_L(\cV_1(\g)) \subseteq \cV_1(\h)$ and so will not explicitly say this from now on.  This is necessary for the $L$ to be Lipschitz.  It follows then that
\begin{align}
  T_L(\cV_j(\g)) \subseteq \cV_j(\h), \qquad \forall j. \label{e:layer-inclusion}
\end{align}
In the exponential coordinates, we then have that
\begin{align*}
  L(\exp(g_1 + ... + g_r)) = \exp(T_L(g_1) + ... + T_L(g_r)).
\end{align*}
Note that $T_L$ is injective if and only if $L$ is injective.

Let $S^{k-1}$ be the unit sphere of $\cV_1$.  Then for any $x \in G$ and $v \in S^{k-1}$, we can isometrically embed $\R$ into $G$ via the map
\begin{align*}
  t \mapsto x e^{tv}.
\end{align*}
The images of these lines are called {\it horizontal lines}.

Finally as we have identified $G$ with $\R^n$, we can speak of the Lebesgue measure.  From looking at the Jacobians of the BCH formula and the dilation automorphism, we get that $\cL^n$, the Lebesgue measure on $\R^n$, is left-invariant and satisfies the identity $|\delta_\lambda(E)| = \lambda^N |E|$ where
\begin{align}
  N = \sum_{i=1}^r i \cdot \dim(\cV_i) \label{e:homogeneous-dimension}
\end{align}
is the {\it homogeneous dimension} of $G$.  As $\Hd^N$ is also a left-invariant $N$-homogeneous measure, by uniqueness of the Haar measure, we have that $\Hd^N$ and $\cL^n$ are multiples of each other.  Here, when we write $|E|$ for a set $E$, we mean the Lebesgue measure of $E$.  The Hausdorff dimension of a Carnot group is exactly its homogeneous dimension.

As $|B(x,r)| = cr^N$ for some $c > 0$ depending only on $G$, we have by basic packing arguments that $G$ is metrically doubling, that is, there exists some $M > 0$ depending only on $G$ so that for each $x \in G$ and $r > 0$, there exists $\{x_i\}_{i=1}^m$ where $m \leq M$ so that
\begin{align*}
  B(x,r) \subseteq \bigcup_{i=1}^m B(x_i,r/2).
\end{align*}
The following theorem of Christ says that such a space contains a collection of partitions that behave like dyadic cubes.
\begin{theorem}[Theorem 11 of \cite{christ}] \label{t:cubes}
  There exists a collection of open subsets $\Delta := \{Q_\omega^k \subset G : k \in \Z, \omega \in I_k\}$ and $\tau > 1,\eta > 0, C < \infty$ such that
  \begin{enumerate}[(a)]
    \item $\left| G \backslash \bigcup_\omega Q_\omega^k \right| = 0, \qquad \forall k \in \Z$.
    \item If $k \geq j$ then either $Q_\alpha^j \subset Q_\omega^k$ or $Q_\alpha^j \cap Q_\omega^k = \emptyset$.
    \item For each $(j,\alpha)$ and each $k > j$ there exists a unique $\omega$ such that $Q_\alpha^j \subset Q_\omega^k$.
    \item For each $Q_\omega^k$ there exists some $z_{Q_\omega^k}$ so that
    \begin{align*}
      B(z_{Q_\omega^k}, \tau^{k-1}) \subseteq Q_\omega^k \subseteq B(z_{Q_\omega^k}, \tau^{k+1}).
    \end{align*}
    \item $|\{x \in Q_\alpha^k : d(x,G \backslash Q_\alpha^k) \leq t \delta^k\}| \leq C t^\eta |Q_\alpha^k|$. \label{e:small-boundaries}
  \end{enumerate}
\end{theorem}
We let $j : Q^k_\omega \mapsto k$ denote the scale of each cube and $\ell : Q \mapsto \tau^{j(Q)}$ be its approximate diameter.  We also let $\Delta_k := \{Q_\omega^k : \omega \in I_k\}$ and given a cube $S \in \Delta$, let $\Delta(S) = \{Q \in \Delta : Q \subseteq S\}$.

Given a map $f : G \to H$ between two Carnot groups, we can define the derivative of $f$ at $p \in G$ as the map
\begin{align*}
  Df_p(u) := \lim_{\lambda \to 0} \delta_{1/\lambda}( f(p)^{-1} f(p\delta_\lambda(u))).
\end{align*}
Pansu proved in \cite{pansu}*{Corollary 3.3 and Proposition 4.1} that for every Lipschitz $f : G \to H$ and almost every $p \in G$, $Df_p : G \to H$ exists and is a Lipschitz homomorphism.  This differentiability theorem was later generalized to Lipschitz functions defined only on measurable subsets of $G$ by Magnani in \cite{magnani}.  Using the Pansu derivative, Magnani was also able to prove in the same paper an area formula.  Specifically if $N$ is the homogeneous dimension of $G$ and we define the Jacobian of $Df_x$ as
\begin{align*}
  J_N(Df_x) := \frac{\Hd^N(Df_x(B(0,1)))}{\Hd^N(B(0,1))},
\end{align*}
then for any measurable $A \subseteq G$ and Lipschitz $f : A \to H$, we have
\begin{align}
  \int_A J_N(Df(x)) ~d\Hd^N(x) = \int_H \# f^{-1}(y) ~d\Hd^N(y). \label{e:area-formula}
\end{align}

\subsection{A nondiscretizable Carnot group}
We recall Definition 2.11 of \cite{meyerson} of discretizability of a Carnot group.  Let $G$ be a Carnot group whose Lie algebra $\g$ admits the stratification
\begin{align*}
  \g = \bigoplus_{j=1}^r \cV_j,
\end{align*}
and let $m_j = \dim \cV_j$.  We say $G$ is {\it discretizable} if for each $j \in \{1,...,r\}$ there exist a collection of vectors
\begin{align*}
  \{X_{(j,i)}\}_{i=1}^{m_j} \in \cV_j
\end{align*}
such that $\{X_{(j,i)}\}_{i=1}^{m_j}$ spans $\cV_j$ and if we let $G' = \langle \{\exp(X_{(1,i)})\}_{i=1}^{m_1} \rangle$ and $G_j = \exp\left( \bigoplus_{k=j}^r \cV_k \right)$, then $G'$ is a discrete subgroup and
\begin{align}
  G' \cap G_j = \langle \{\exp(X_{(k,i)}) \}_{1 \leq i \leq m_k, j \leq k \leq r} \rangle. \label{e:cocompact}
\end{align}
In other words, a group $G$ is discretizable if there exists a basis of horizontal elements that generate a discrete subgroup spanning all of $G$.  Recall that the biLipschitz decomposition result of \cite{meyerson} required that the domain Carnot group be discretizable.  We now prove that not all Carnot groups are discretizable.

\begin{proposition}
  There exists a Carnot group that is not discretizable.
\end{proposition}

\begin{proof}
  We let $G$ be the Carnot group that has the stratified Lie algebra $\g$ of step 6 that we now describe.  The horizontal layer $\cV_1$ is two dimensional and spanned by two vectors $X$ and $Y$.  The other layers are 1-dimensional, and we let $Z_i$ be vectors spanning $\cV_i$ for $i \in \{2,...,6\}$.  We define the relations $[X,Y] = Z_2$,
  \begin{align*}
    [X,Z_i] = t_{i+1} Z_{i+1} \qquad \text{and} \qquad [Y,Z_i] = Z_{i+1}, \qquad i \in \{2,...,5\}.
  \end{align*}
  Here, $t_3,...,t_6$ are real numbers that we choose later.
  
  Suppose $G$ is discretizable, and let $G'$ be the discrete subgroup as in the definition.  Then $G'$ is generated by $g = \exp(aX + bY)$ and $h = \exp(cX + dY)$ for $a,b,c,d \in \R$, and we may suppose that $ad - bc = 1$.  By assumption, there exists some $s_i \neq 0$ such that $u_i = \exp(s_i Z_i)$ for $i \in \{2,...,5\}$ are elements of $G'$.  We then have by the BCH formula that
  \begin{align*}
    gu_ig^{-1}u_i^{-1} &= \exp((at_{i+1}+b)s_i Z_{i+1} + Y_i) \in G' \cap G_{i+1}, \\
    hu_ih^{-1}u_i^{-1} &= \exp((ct_{i+1}+d)s_i Z_{i+1} + Y_i') \in G' \cap G_{i+1},
  \end{align*}
  where $Y_i,Y_i' \in \bigoplus_{k=i+2}^6 \cV_k$.
  
  In order for $G'$ to be discrete, we must have that $\frac{at_i+b}{ct_i+d} \in \Q$ for all $i \in \{3,...,6\}$.  Indeed, let us suppose $G'$ is discrete.  Let
  \begin{align*}
    v &= gu_5g^{-1}u_5^{-1} = \exp((at_6+b) s_5 Z_6), \\
    w &= hu_5h^{-1}u_5^{-1} = \exp((ct_6+d) s_5 Z_6),
  \end{align*}
  which are elements of $G'$.  Thus, we have that
  \begin{align*}
    \{ v^p w^q : p,q \in \Z \} = \{ \exp( (p (at_6+b) + q (ct_6 + d)) s_5 Z_6 ) : p,q \in \Z\} \subset G'
  \end{align*}
  If $\frac{at_6 + b}{ct_6 + d} \notin \Q$, then we get that $\{p (at_6 + b) + q (ct_6 + b) : p,q \in \Z\}$ is dense in $\R$ and then $\{v^pw^q : p,q \in \Z\}$ is dense in $\exp(\cV_6)$, which contradicts the discreteness of $G'$.  We can repeat the same argument to show then that
  \begin{align*}
    \{ \exp( (p (at_i+b) + q (ct_i + d)) s_{i-1} Z_i + W_{p,q} ) : p,q \in \Z\} \subset G'
  \end{align*}
  for all $i$ where $W_{p,q} \in \bigoplus_{k=i+1}^6 \cV_k$ lies in a bounded region.  Here, one uses \eqref{e:cocompact} to show that such a $W_{p,q}$ can be chosen.  Then, by the same argument, we get that $\frac{at_i + b}{ct_i +d} \in \Q$ for all $i$.
  
  Note that if the map
  \begin{align*}
    \phi : x \mapsto \frac{ax + b}{cx + d}
  \end{align*}
  takes three distinct rationals to rationals and $ad - bc = 1$, then $a,b,c,d$ are all rational and so $\phi$ takes all rationals to rationals (and possibly infinity) and irrationals to irrationals.  Thus, if we specify $t_3,...,t_6$ to be three distinct rational numbers and an irrational one, then $G'$ cannot be discrete for any such choice of $a,b,c,d$ and so $G$ is not discretizable.
\end{proof}

\section{Distortion and nets}
From here on, we let $G$ and $H$ be two Carnot groups with Lie algebras $\g$ and $\h$ of step $r$ and $s$, respectively.  Let $L : G \to H$ be a homomorphism.  As mentioned before, one can lift this homomorphism via the exponential map to a linear transform $T_L : \g \to \h$.

Our first lemma says that a homomorphism that collapses points does so on the layers.
\begin{lemma} \label{l:inradius}
  Let $G$ and $H$ be as above and $L : G \to H$ be a homomorphism such that there exists $g \in G$ so that $N_\infty(L(g)) < \epsilon N_\infty(g)$ for some $\epsilon > 0$.  Then there exists $j \in \{1,...,r\}$ and some $v \in \cV_j(\g)$ for which $|v| = 1$ and $N_\infty(L(e^v)) < \epsilon$.
\end{lemma}

\begin{proof}
  Suppose the conclusion is false.  Then using the definition of $N_\infty$ and the fact that $L = \exp \circ T_L \circ \exp^{-1}$, we have
  \begin{align}
    |T_L(u)|^{1/i} \geq \epsilon |u|^{1/i}, \qquad \forall u \in \cV_i(\g), \forall i \in \{1,...,r\}. \label{e:short-layer-vectors}
  \end{align}
  Let $g = \exp(g_1 + ... + g_r) \in G$.  Then there exists some $k \in \{1,...,r\}$ so that $|g_k|^{1/k} = N_\infty(g)$.  We have that $L(\exp(g_1 + ... + g_r)) = \exp(T_L(g_1) + ... + T_L(g_r))$ and by the definition of $N_\infty$ of $H$, we have that
  \begin{align*}
    N_\infty(\exp(T_L(g_1) + ... + T_L(g_r))) = \max_i |T_L(g_i)|^{1/i} \geq |T_L(g_k)|^{1/k} \overset{\eqref{e:short-layer-vectors}}{\geq} \epsilon |g_k|^{1/k} = \epsilon N_\infty(g).
  \end{align*}
  As $g$ was arbitrary, this contradicts our assumption.
\end{proof}

Assume that $\dim(\g) > \dim(\h)$ and let $f : B(x,2R) \to H$ be Lipschitz.  We then get that any Pansu-derivative $Df_p$ cannot be injective as $T_{Df_p}$ cannot be.  Thus, by \eqref{e:layer-inclusion} and Lemma \ref{l:inradius}, there must exist some $\cV_j(\g)$ such that
\begin{align*}
  \dim(T_{Df_p}(\cV_j(\g))) < \dim(\cV_j(\g)).
\end{align*}
By the definition of homogeneous dimension \eqref{e:homogeneous-dimension}, we get then that $Df_p(G)$ has homogeneous/Hausdorff dimension less than $N$ and so $J_N(Df_p) = 0$.  As this holds for all $p \in G$ in a set of full measure where the Pansu-derivative exists, we get by the area formula that
\begin{align*}
  \Hd^N_\infty(f(B(x,2R))) \leq \Hd^N(f(B(x,2R))) \overset{\eqref{e:area-formula}}{=} 0.
\end{align*}
There is then nothing to prove for Theorem \ref{t:main}.  Thus, from now on, we can and will assume that $n = \dim(\g) \leq \dim(\h) = m$.

Our next lemma says that right translation does not distort coordinates too much.
\begin{lemma} \label{l:box-distort}
  There exists some $C_2 > 0$ depending only on $H$ so that if $g,h \in H$ are such that $N_\infty(g) \leq 1$ and $N_\infty(h) \leq \epsilon$ for some $\epsilon > 0$, then $|(g \cdot h)_i - g_i| \leq C_2\epsilon$ for all $i \in \{1,...,s\}$.
\end{lemma}

\begin{proof}
  This follows from the BCH formula.  Fix some $i \in \{1,...,s\}$.  Then
  \begin{align*}
    |(g \cdot h)_i - g_i| = |h_i + P_i(g_1,...,g_{i-1},h_1,...,h_{i-1})|.
  \end{align*}
  As $|h_i| \leq \epsilon^i \leq \epsilon$, it suffices to show that $|P_i| \leq C\epsilon$ for some $C > 0$.  As $P_i$ is a polynomial of nested Lie brackets where the number of brackets and the coefficients are bounded by some number depending only on $i$, it further suffices to bound each nested Lie bracket by $C\epsilon$.  Let $[x_1,[x_2,...,[x_{\ell-1},x_\ell]...]]$ be one such term.  By the BCH formula, one of $x_{\ell-1}$ and $x_\ell$ is a coordinate of $h$.  Thus, as $|g_j| \leq 1$ and $|h_j| \leq \epsilon^j \leq \epsilon$ for all $j \geq 1$, we have that
  \begin{align*}
    |[x_1,[x_2,...,[x_{\ell-1},x_\ell]...]]| \leq \prod_{j=1}^\ell |x_j| \leq \epsilon.
  \end{align*}
  This finishes the proof.
\end{proof}

We can now prove the main result of this section which says that if a homomorphism collapses points, then we can cover the homomorphic image of a ball by only a few small balls.  In the proof (and statement), balls $B(x,r)$ will be balls in the $d_\infty$ metric of their respective Carnot groups.  We let cubes of the form $[a,b]^k$ denote the exponential images in $H$ of these sets in $\h$ which we have identified with a Euclidean space.  We also let $B_{\R^n}(0,s)$ denote the exponential images in $H$ of the corresponding Euclidean balls in $\h$.

\begin{lemma} \label{l:nets-lemma}
  There exists some $C_3 > 0$ depending only on $G$ and $H$ so that if $\epsilon > 0$ and $L : (G,d_\infty) \to (H,d_\infty)$ is a 1-Lipschitz homomorphism such that there exists $g \in G$ so that $N_\infty(L(g)) < \epsilon N_\infty(g)$, then for every $x \in G$ and $\ell \geq 0$, there exist points $\{x_i\}_{i=1}^{N_\epsilon} \subset L(B(x,\ell))$ with $N_\epsilon \leq C_3 \epsilon^{1-N}$ so that
  \begin{align*}
    L(B(x,\ell)) \subset \bigcup_{i=1}^{N_\epsilon} B(x_i,\epsilon \ell).
  \end{align*}
  Here, $N$ is the Hausdorff/homogeneous dimension of $G$.
\end{lemma}

\begin{proof}
  By homogeneity and left-invariance of the metric, we may suppose that $\ell = 1$ and $x = 0$.  Note that $L(G)$ is a Lie subgroup of $H$ with Lie algebra $T_L(\g) \subseteq \h$.  First suppose $L$ is injective.  As we are identifying $H$ with its Lie algebra $\h$, which we can also view as $\R^m$, $L(G)$ can also be identified with a linear subspace $\R^n \subseteq \R^m$.  Note then that $L(B(0,1))$ is a symmetric convex subset of $\R^n$.  From Lemma \ref{l:inradius}, we have that the inradius of $L(B(0,1))$ is less than $\epsilon$.  As $L$ is 1-Lipschitz, we also have that $L(B(0,1))$ is contained in $[-1,1]^n$.
  
  One can see from the Jacobian of the dilation that $\cL^n(B(0,s) \cap L(G)) = s^N \cL^n(B(0,1) \cap L(G))$.  In addition, as $L(G)$ is a Lie subgroup of $H$, we have by the BCH formulas that left translation by an element of $L(G)$ preserves the volume form of $\R^n$.  Thus, $\cL^n$ is a left invariant measure on $L(G)$ that is $N$-homogeneous with respect to dilations.
  
  Note that the $d_\infty$ metric of $H$ satisfies $\left[ -m^{-1/2}, m^{-1/2} \right]^m \subseteq B(0,1)$.  By simple Euclidean geometry, we have that
  \begin{align*}
    \left[-m^{-1/2},m^{-1/2} \right]^n \subseteq \left[ -m^{-1/2}, m^{-1/2} \right]^m \cap \R^n
  \end{align*}
  where the cube on the left hand side is written using the coordinates of $\R^n$ (with the induced metric from $\R^m$).  We also have that $B(0,1) \subset [-1,1]^m$.  Again, by Euclidean geometry, we have
  \begin{align*}
    \left[-m^{1/2},m^{1/2} \right]^n \supseteq [-1,1]^m \cap \R^n.
  \end{align*}
  Altogether, we have that
  \begin{align*}
    \left[ - m^{-1/2},m^{-1/2} \right]^n \subseteq B(0,1) \cap L(G) \subseteq \left[ -m^{1/2}, m^{1/2} \right]^n.
  \end{align*}
  Thus, we see that there exists some $c \in [2^nm^{-n/2},2^nm^{n/2}]$ so that $\cL^n(B(0,1) \cap L(G)) = c$.  It then follows from the fact that $\cL^n$ is a left-invariant $N$-homogeneous measure on $L(G)$ that if $x \in L(G)$, then
  \begin{align}
    \cL^n(B(x,r) \cap L(G)) = cr^N. \label{e:plane-ball-volume}
  \end{align}
  
  Take $\{x_i\}_{i=1}^{N_\epsilon}$ to be a maximal $\epsilon$-separated net in $L(B(0,1))$.  We claim that $N_\epsilon \leq C_3 \epsilon^{1-N}$ for some $C_3 > 0$.  The sets $\{B(x_i,\epsilon/4)\}_{i=1}^{N_\epsilon}$ are disjoint in $L(B(0,1))$.  Indeed, if there exists some intersection $z \in B(x_i,\epsilon/4) \cap B(x_j,\epsilon/4)$, then
  \begin{align*}
    d_\infty(x_i,x_j) \leq d_\infty(x_i,z) + d_\infty(z,x_j) \leq \frac{\epsilon}{2},
  \end{align*}
  contradicting our assumption that $\{x_i\}$ were $\epsilon$-separated.  Lemma \ref{l:box-distort} gives that $S = \bigcup_{i=1}^{N_\epsilon} (B(x_i,\epsilon/4) \cap L(G))$ is contained in an $C_2 \epsilon$-neighborhood (with respect to the Euclidean metric of $\R^n$) of $L(B(0,1))$.  Thus,
  \begin{align*}
    S \subset L(B(0,1)) + B_{\R^n}(0,C_2\epsilon).
  \end{align*}
  As $L(B(0,1))$ is a symmetric convex set, $L(B(0,1)) + B_{\R^n}(0,C_2 \epsilon)$ is also a symmetric convex set which has inradius at most $(C_2 + 1)\epsilon$ and is contained in $[-1,1]^n + B_{\R^n}(0,C_2\epsilon) \subseteq [-2,2]^n$ (assuming $\epsilon < 1/C_2$, which we can).  Then it follows from Euclidean geometry that
  \begin{align*}
    \cL^n(L(B(0,1)) + B_{\R^n}(0,C_2 \epsilon)) \leq C \epsilon
  \end{align*}
  for some $C > 0$ depending only on $n$ and $C_2$.  Thus,
  \begin{align*}
    N_\epsilon c(\epsilon/4)^N \overset{\eqref{e:plane-ball-volume}}{=} \cL^n(S) \leq \cL^n(L(B(0,1)) + B_{\R^n}(0,C_2 \epsilon)) \leq C\epsilon
  \end{align*}
  and so $N_\epsilon \leq 4^Nc^{-1}C \epsilon^{1-N}$, which proves our claim.
  
  It now remains to prove that $L(B(0,1)) \subset \bigcup_{i=1}^{N_\epsilon} B(x_i,\epsilon)$.  This follows from general packing arguments principles from metric geometry.  Suppose not.  Then there exists some $z \in L(B(0,1))$ such that $d_\infty(z,x_i) \geq \epsilon$ for all $i$.  Thus, $\{x_i\}_{i=1}^{N_\epsilon} \cup \{z\}$ is an even larger $\epsilon$-separated net in $L(B(0,1))$, contradicting our maximality assumption.  This proves the lemma in the case $L$ is injective.

  Now assume $L$ is not injective.  Then by Lemma \ref{l:inradius}, there exists some $j \in \{1,...,r\}$ so that $\dim T_L(\cV_j(\g)) < \dim \cV_j(v)$ and $T_L(\g)$ has dimension $n' \leq n-1$ in $\R^m$.  One can then see from the Jacobian of the dilation that $\cL^{n'}(B(x,r) \cap L(G)) = cr^{N'}$ for some $c > 0$ and $N' \leq N - 1$.  This allows us to continue the argument starting from \eqref{e:plane-ball-volume} and the $N'$-homogeneity of $\cL^{n'}$ allows the arguments to still work.
\end{proof}

\begin{remark} \label{r:equivalent-nets}
  Note that the result of Lemma \ref{l:nets-lemma} still holds if we pass from $d_\infty$ of $G$ and $H$ to biLipschitz equivalent (semi)metrics.  The constant $C_3$ will change only by a factor controlled by the biLipschitz equivalence.
\end{remark}

\section{Weak biLipschitzness}
In this section, we let first let $f : [a,b] \to X$ be a 1-Lipschitz map where $(X,d)$ is an arbitrary metric space.  We now recall some terminology from \cite{li}.  Given some $p \geq 1$ and $x,y \in \R$, we let
\begin{multline}
  \partial_f^{(p)}(x,y) = \frac{1}{2} \left[ \left( \frac{d(f(x),f((x+y)/2))}{|y-x|/2} \right)^p + \left( \frac{d(f((x+y)/2),f(y))}{|y-x|/2} \right)^p \right] \\
  - \left( \frac{d(f(x),f(y))}{|y-x|} \right)^p. \label{e:partial-defn}
\end{multline}
We will not actually need that $d$ satisfies the triangle inequality in this section, just that
\begin{align}
  \partial_f^{(p)}(x,y) \geq 0, \qquad \forall x,y \in X, \label{e:partial-positive}
\end{align}
which follows when $d$ is a metric by Jensen's inequality and the triangle inequality.  For us, we will be using a semimetric and so there is no {\it a priori} guarantee that $\partial_f^{(p)}(x,y)$ is positive.  This will instead follow from the properties of our carefully chosen semimetric as will be done in Lemma \ref{l:coarse-diff}.  We now define the quantity
\begin{align}
  \alpha_f^{(p)}([a,b]) = \frac{1}{2(b-a)^2} \underset{a \leq x < y \leq b}{\iint} \partial_f^{(p)}(x,y) ~dx ~dy. \label{e:alpha-defn}
\end{align}
This integral should be thought of as a $p$-convex analogue of the triple integral quantity displayed in equation (1.5) of \cite{schul-AR}.  There, the quantity integrated is the excess of the triangle inequality whereas here we are integrating the excess of the $p$-parallelogram inequality given in \eqref{e:partial-defn}.  In \cite{schul-AR}, a weighted sum of the triple integral on a subset of a 1-rectifiable set is shown to be bounded.  In a similar fashion, if $X$ satisfies \eqref{e:partial-positive}, then a weighted sum of $\alpha_f^{(p)}([a,b])$ is bounded for a Lipschitz map $f : [a,b] \to X$ (see \cite{li}*{Lemma 3.2}).

Now let $G$ be a Carnot group of Hausdorff dimension $N$, $k = \dim(\cV_1(\g))$, and $f : G \to H$ be a Lipschitz function.  We will equip $G$ with the $d_\infty$ metric but will define the metric of $H$ later.  We can extend the definition of $\alpha$ to the Christ cubes of $G$.  Let $L \geq 1$ and $Q \in \Delta$.  Then we can define
\begin{align*}
  \alpha_f^{(p)}(Q,L) = \frac{1}{(L \ell(Q))^{N-1}} \int_{S^{k-1}} \int_{z_Q(G \ominus v)} &\chi_{\{x \exp(\R v) \cap B(z_Q,L\ell(Q)) \neq \emptyset\}} \times \\
  &\alpha^{(p)}_f(x\exp(\R v) \cap B(z_Q,3L\ell(Q))) ~dx ~d\mu(v).
\end{align*}
Here, $z_Q(G \ominus v)$ is the left translate of the exponential image of the subspace of $\g$ orthogonal to $v$ by $z_Q$.  Integration in $x$ is with respect to the $N-1$ dimensional Hausdorff measure $\Hd^{N-1}$ and in $v$ is with respect to the probability measure on the unit sphere of $\cV_1(\g)$.

One may be worried that there is no guarantee that $x \exp(\R v) \cap B(z_Q,3L\ell(Q))$ is connected.  We simply specify that it be the unique connected subset $I$ that contains the subset $x \exp(\R v) \cap B(z_Q,3L\ell(Q))$.  See Lemma 3.3 and the following discussion of \cite{li}.

The following proposition shows that the $\alpha(Q,L)$ are Carleson summable in a cube $S$ with constant depending on $L$.
\begin{proposition} \label{p:alpha-L-sum}
  Let $G$ be a Carnot group and $X$ be a semimetric space satisfying \eqref{e:partial-positive}.  For each $L \geq 1$, there exists some $C_1 = C_1(L) > 0$ depending on $L$ so that if $f : G \to X$ is 1-Lipschitz and $S \in \Delta$, then we have
  \begin{align}
    \sum_{Q \in \Delta(S)} \alpha_f^{(p)}(Q,L) |Q| \leq C_1(L) |S|. \label{e:alpha-sum}
  \end{align}
\end{proposition}

\begin{proof}
  The proof is essentially that of Proposition 3.5 of \cite{li} with $\epsilon = 0$ and $m = \infty$.  Note that the second parameters of the $\alpha_f^{(p)}$ quantities defined here and in \cite{li} are not the same.  One makes the straightforward changes to account for the scaling $L$, which only changes the constants in the proof by an amount controllable by $L$.  Notably, the constant $C_2$ of that proof will change by an amount depending on $L$ and $G$.
\end{proof}

The requirement of \eqref{e:partial-positive} is used to show that $\alpha_f^{(p)}(Q,L)$ are all nonnegative.  This is then used in the next section in conjunction with \eqref{e:alpha-sum} to show that there cannot be ``too many'' $Q \in \Delta(S)$ for which $\alpha_f^{(p)}(Q,L)$ is large.

The following lemma will give the needed metric for $H$.
\begin{lemma} \label{l:coarse-diff}
  There exists $p \geq 2$, $C,\alpha,\beta,\gamma > 0$, and a homogeneous norm $N$ on $H$ whose induced semimetric $d_H$ satisfies \eqref{e:partial-positive} so that if we define $\alpha_f^{(p)}$ with respect to $d_H$, then the following property holds.  Let $\epsilon \in (0,1/2)$ and $f : G \to H$ be Lipschitz.  If
  \begin{align*}
    \alpha_f^{(p)}(Q,\gamma\epsilon^{-\alpha}) \leq e^{-\epsilon^{-\beta}} \|f\|_{lip}^p,
  \end{align*}
  then there exists a homomorphism $L : G \to H$ and $g \in H$ so that for all $x \in B(z_Q,100 \diam(Q))$, we have
  \begin{align}
    \frac{d_H(f(x),g \cdot L(x))}{\diam(Q)} \leq \epsilon \|f\|_{lip}. \label{e:eps-close-lemma}
  \end{align}
\end{lemma}

\begin{proof}
  The norm is the one given by Proposition 7.2 of \cite{li}.  That $\partial_f^{(p)}(x,y) \geq 0$ follows from equation (89) of that proposition (using the norm $N$ to induce the corresponding metric $d_H(x,y) = N(x^{-1}y)$), which shows that for $p = 2s!$, we have that there exists some $C > 0$ so that
  \begin{align*}
    2^{-p} |x-y|^p \partial_f^{(p)}(x,y) \geq C(|\pi(f(x)) - \pi(f(y))|^p + NH(f(x)^{-1}f(y))^p) \geq 0.
  \end{align*}
  Here, $\pi : H \to \R^{\dim \cV_1(\h)}$ is the 1-Lipschitz homomorphic projection of the first layer and $NH(g) = d_H(g,\exp(g_1))$ can be thought of as a measure of how nonhorizontal $g$ is.
  
  The result then follows from using Lemma 6.12 together with equation (55) of the same paper.  Specifically, equation (55) says that there is some constant $C' > 0$ depending only on $H$ so that $\beta_f^{(p)} \leq C' \alpha_f^{(p)}$.  Here, we remind the reader that parameters of the $\alpha_f^{(p)}$ are not the same in this paper as in \cite{li}.  What is defined as $\alpha_f^{(p)}(Q,L)$ in this paper is equivalent to $\alpha_f^{(p)}(\frac{L}{2}Q,0)$ in \cite{li}.  The $\epsilon$ stays the same, and we take $\psi = 0$ and $\text{Lip}_f(\psi) = \|f\|_{lip}$ in Lemma 6.12 as we are dealing with Lipschitz functions (for the definition of $\text{Lip}_f$, see the last displayed equation of p. 4621 of \cite{li}).  All the constants used in Lemma 6.12 are given by the choice of $d_H$.  Lemma 6.12 then derives from a bound on (our current version of) $\alpha_f^{(p)}(Q,1)$ a statement like \eqref{e:eps-close-lemma} on a small subball of $Q$ centered around $x_Q$.  Here, the $\beta$ in the statement of this lemma will depend on $\beta_0,r,s,\alpha_1$ used in Lemma 6.12.  It follows easily then that a similar bound on a dilate of $Q$---as in our hypothesis---gives our needed result on $B(z_Q,100 \diam(Q))$.
\end{proof}

From now on, we now endow $H$ with the metric $d_H$ induced by the norm of Lemma \ref{l:coarse-diff}.  Note that we never stated that $d_H$ satisfies the triangle inequality.  Thus, let $C_Q \geq 1$ be such that
\begin{align*}
  d_H(x,y) \leq C_Q(d_H(x,z) + d_H(z,y)), \qquad \forall x,y,z \in H.
\end{align*}

The properties of cubes given by Theorem \ref{t:cubes} easily imply that there exists some $b \in (0,1/10)$ depending only on $G$ so that if $x,y \in G$ and $Q$ is the smallest cube containing $x$ such that $y \in 2Q$ (recall the notation from \eqref{e:lambda-E}), then
\begin{align}
  d(x,y) \geq 10b \diam(Q). \label{e:b-defn}
\end{align}
We fix this $b$ for the rest of the paper.

The following lemma is the main result of this section.  It says that if a cube $Q$ has an image under $f$ with large Hausdorff content but small $\alpha_f^{(p)}$, then it must push far away points apart.  A function that satisfies the result of this lemma is sometimes called weakly biLipschitz.
\begin{lemma} \label{l:weak-bilip}
  There exists some $c_1 > 0$ depending only on $G$ and $H$ so that for each $\delta > 0$, if $f : G \to H$ is 1-Lipschitz, $Q \in \Delta$ so that
  \begin{align*}
    \Hd^N_\infty(f(Q)) > c_1 \delta |Q|
  \end{align*}
  and
  \begin{align*}
    \alpha_f^{(p)}(Q,\gamma\delta^{-\alpha}) < e^{-\delta^{-\beta}},
  \end{align*}
  then for each $x,x' \in 2Q$ such that
  \begin{align}
    d(x,x') > b \diam(Q) \label{e:x-x'-sep}
  \end{align}
  we have that
  \begin{align*}
    d_H(f(x),f(x')) > \delta d(x,x'),
  \end{align*}
\end{lemma}

\begin{proof}
  By our assumption of $\alpha_f^{(p)}$ and Lemma \ref{l:coarse-diff}, there exists a homomorphism $L : G \to H$ and $g \in H$ so that for all $x \in B(z_Q, 100 \diam(Q))$ we have
  \begin{align}
    d_H(f(x),g \cdot L(x)) \leq \delta \diam(Q). \label{e:eps-close}
  \end{align}
  Note that $x,x' \in 2Q \subseteq B(z_Q, 100 \diam(Q))$.  If for all $z \in G$, we have $d_H(L(z),0) \geq 4b^{-1}C_Q^2 \delta d(z,0)$, then we are done because
  \begin{multline*}
    d_H(f(x),f(x')) \geq \frac{1}{C_Q^2} d_H(L(x),L(x')) - d_H(f(x),L(x)) - d_H(f(x'),L(x')) \\
    \overset{\eqref{e:eps-close}}{\geq} \frac{4\delta}{b} d(x,x') - 2\delta \diam(Q) \overset{\eqref{e:x-x'-sep}}{\geq} (4\delta - 2\delta) \diam(Q) \geq 2\delta \diam(Q) \geq \delta d(x,x').
  \end{multline*}
  In the last inequality, we used the fact that $x,x' \in 2Q$.
  
  Thus, we may suppose that there exists some $z \in G$ so that
  \begin{align*}
    d_H(L(z),0) < 4b^{-1}C_Q^2 \delta d(z,0).
  \end{align*}
  By Lemma \ref{l:nets-lemma} and Remark \ref{r:equivalent-nets}, there exists some $C_4 > 0$ depending only on $G$ and $H$ and points $\{x_i\}_{i=1}^{N_\delta} \subset L(B(z_Q, \tau\ell(Q)))$ so that $N_\delta \leq C_4 \delta^{1-N}$ and
  \begin{align*}
    L(Q) \subseteq L(B(z_Q, \tau\ell(Q))) \subset \bigcup_{i=1}^{N_\delta} B(x_i, 4b^{-1} C_Q^2 \delta \tau \ell(Q)).
  \end{align*}
  Note then that
  \begin{align}
    f(Q) \subseteq \bigcup_{i=1}^{N_\delta} B(x_i, C_Q(2 + 4b^{-1} C_Q^2)\delta \tau \ell(Q) ). \label{e:fQ-contain}
  \end{align}
  Indeed, we have that as $Q \subseteq B(z_Q, \tau\ell(Q))$, we have for any $x \in Q$ that there exists some $i$ so that $d_H(L(x),x_i) \leq 4C_Q^2b^{-1} \delta \tau \ell(Q)$.  Thus,
  \begin{multline*}
    d_H(f(x),x_i) \leq C_Q( d_H(f(x),L(x)) + d_H(L(x),x_i)) \\
    \overset{\eqref{e:eps-close}}{\leq} C_Q \left( 2\delta \tau \ell(Q) + \frac{4C_Q^2}{b} \delta \tau \ell(Q) \right) \leq C_Q(2 + 4b^{-1}C_Q^2)\delta \tau \ell(Q).
  \end{multline*}
  It follows from \eqref{e:fQ-contain} that there exists some $C_5 > 0$ depending only on $G$ and $H$ so that
  \begin{align*}
    \Hd^N_\infty(Q) \leq N_\delta \left[ 2C_Q(2 + 4b^{-1} C_Q^2) \tau\right]^N \delta \ell(Q)^N \leq C_5 \delta |Q|.
  \end{align*}
  Here we used the fact that $\ell(Q)^N$ is comparable to $|Q|$.  Choosing $c_1$ small enough, we get a contradiction.
\end{proof}

\section{Proof of main theorem}
The proof is relatively standard and follows the arguments of \cite{david-semmes}*{p. 867} \cite{jones}*{p. 119} (see also \cite{david-wavelets}*{Lemma 8.4}).  As proving the theorem for one homogeneous (semi-)metric on $H$ immediately implies the same result for all homogeneous (semi-)metrics on $H$ with just modified constants, we are free to assign any homogeneous (semi-)metric to $H$.  We will equip $H$ with the homogeneous $d_H$ semimetric from Lemma \ref{l:coarse-diff}.  We remind the reader that $d_H$ is not a metric but that this is fine since all the lemmas and propositions used in this section do not require that it is a metric.  By scale invariance, we may suppose that $R = 1$.

We first specify an $\epsilon > 0$ small enough, depending only on $G$ and $\delta$, so that $|B(x,1) \backslash B(x,(1-\epsilon))| < \delta$.  Then as $f$ is 1-Lipschitz, we get that
\begin{align*}
  \Hd_\infty^N(f(B(x,1) \backslash B(x,(1-\epsilon)))) \leq \Hd^N(f(B(x,1) \backslash B(x,(1-\epsilon)))) < \delta.
\end{align*}
We can now specify a $j$ small enough depending only on $\delta$ and $G$ so that if $Q \in \Delta_j$ is a cube such that $Q \cap B(x,1-\epsilon) \neq \emptyset$, then all the horizontal line segments $x \exp(\R v) \cap B(z_Q,3L\ell(Q))$ needed in the calculation of $\alpha_f^{(p)}(Q,\gamma \delta^{-\alpha})$ are contained in $B(x,1)$.  The number and collective volume of such cubes are bounded by constants depending only on $\delta$ and $G$.  Thus, if we show the result for each of these cubes, we can take the union of all the biLipschitz pieces (of which there is a controlled number) to get our needed biLipschitz decomposition of $f$ on $B(x,1)$.  We now let $S$ be one of these cubes and we will prove the statement of the theorem for $S$ in place of $B(x,R)$.  Due to all the initial work we did, we may assume $f$ is defined on all of $G$.

Define the following families of cubes
\begin{align*}
  \cB_1 &= \{Q \in \Delta(S) : \Hd^N_\infty(f(Q)) < c_1\delta|Q|\}, \\
  \cB_2 &= \{Q \in \Delta(S) : \alpha_f^{(p)}(Q,\gamma \delta^{-\alpha}) \geq e^{-\delta^{-\beta}}\}.
\end{align*}
Given a cube $Q$, let $\hat{Q}$ be the union of cubes of $\Delta_{j(Q)}$ that intersect $2Q$.  Given $L > 0$, we define
\begin{align*}
  R_2 = \left\{ x \in S : \sum_{Q \in \cB_2} \chi_{\hat{Q}}(x) \geq L \right\}.
\end{align*}
We have that
\begin{align*}
  \int_S \sum_{Q \in \cB_2} \chi_{\hat{Q}}(x) ~dx \leq C \sum_{\cB_2} |Q| \leq Ce^{\delta^{-\beta}} \sum_{Q \in \Delta(S)} \alpha_f^{(p)}(Q,\gamma \delta^{-\alpha}) |Q| \overset{\eqref{e:alpha-sum}}{\leq} C_1(\gamma\delta^{-\alpha})Ce^{\delta^{-\beta}} |S|
\end{align*}
Thus, there exists some $L > 1$ depending only on $G$, $H$, and $\delta$ so that $|R_2| < \delta |S|$.  Let $R_1 = \bigcup_{Q \in \cB_1} Q$.  Using the definition of $\cB_1$, we have that
\begin{align*}
  \Hd^N_\infty(f(R_1)) < c_1 \delta \Hd_\infty^N(R_1) \leq c_1 \delta |S|.
\end{align*}
Likewise, as $f$ is 1-Lipschitz and $|R_2| < \delta |S|$, we have that $\Hd_\infty^N(f(R_2)) < \delta |S|$.  Thus, we have that
\begin{align*}
  \Hd_\infty^N(f(R_1 \cup R_2)) < (1 + c_1) \delta |S|.
\end{align*}

It remains to decompose $S \backslash (R_1 \cup R_2)$ into $M$ biLipschitz pieces.  We use the usual encoding scheme, which we will give a sketch of right now.

Let $l \geq 1$ be large enough so that if $Q \in \Delta_k$ and $S \in \Delta_{k+l}$, then $\diam Q < b \diam(S)$.  Then for each $k$ and $Q \in \Delta_k \cap \Delta(S)$, we let $\cF(Q)$ denote the set of cubes $Q' \in \Delta_k \cap \Delta(S)$ such that $Q' \neq Q$ and $Q$ and $Q'$ are both contained in some $\hat{S}$ for some $S \in \Delta_{k+l} \cap \cB_2$.  As $G$ is doubling, we get that there exists some $T \geq 1$ so that $\#\cF(Q) \leq T$ for all $Q \in \Delta(S)$.

Let $A$ be a set of $T+1$ distinct elements.  We will associate to each $Q \in \Delta(S)$ an (possibly empty) ordered string of of characters from $A$ (a word) that we will denote $a(Q)$.  For any $Q \in \Delta$, we let $Q^*$ be the unique parent of $Q$.  The words that we assign will satisfy the following property: $a(S) = \emptyset$, $a(Q) = a(Q^*)$ if $\cF(Q) = \emptyset$, if $\cF(Q) \neq \emptyset$ then $a(Q)$ will be the word $a(Q^*)$ appended with an additional element from $A$ at the end so that if $Q'$ is another cube of $\cF(Q)$ then
\begin{itemize}
  \item $a(Q) \neq a(Q')$ when $a(Q)$ and $a(Q')$ are of equal length,
  \item when $a(Q)$ is shorter than $a(Q')$ then $a(Q)$ does not begin with $a(Q')$,
  \item when $a(Q')$ is shorter than $a(Q)$ then $a(Q')$ does not begin with $a(Q)$.
\end{itemize}
Such an association can be done recursively.  We omit the details but the reader can consult \cite{david-wavelets}*{p. 82} or \cite{jones}*{p. 120} for full details.

Note that if $x \notin R_2$, then the number of cubes $Q$ containing $x$ for which $\cF(Q) \neq \emptyset$ is bounded by $L$.  Thus, there must be some $Q \in \Delta(S)$ containing $x$ for which if $Q' \subset Q$ is any other cube and $Q'$ also contains $x$, then $a(Q') = a(Q)$.  That is, the code stabilizes.  We can then associate to $x$ the word $a(x) = a(Q)$ for this $Q$.  It follows that $a(x)$ is a word of at most $L$ letters.  Thus, we have partitioned $S \backslash (R_1 \cup R_2)$ into at most $(T+1)^L$ measurable sets $\{F_\omega\}_{\omega \in A^{L+1}}$ based on each point's word assignment.  It remains to prove that if $x,y \in F_i$, then $d_H(f(x),f(y)) \geq \delta d(x,y)$.

Let $x,y \in F_\omega$ be two distinct points and let $Q$ be the smallest cube such that $x \in Q$ and $y \in 2Q$.  If $Q \notin \cB_2$, then Lemma \ref{l:weak-bilip} gives us our needed result.  Thus, we may suppose that $Q \in \cB_2$.  We let $Q_0,Q_1 \in \Delta_{j(Q)+l}$ so that $x \in Q_0$ and $y \in Q_1$.  As $d(x,y) \geq 10b \diam(S) > \diam(Q_0)$ by definition of $b$ and $l$, we get that $Q_0 \neq Q_1$.  Thus, $Q_1 \in \cF(Q_0)$ and so by the rules of the assignment of words to cubes, we have that $a(x) \neq a(y)$.  This contradicts our assumption that $x,y \in F_\omega$.
\qed

\begin{remark}
  We may replace $B(x,R)$ in the statement of Theorem \ref{t:main} by any cubes $Q$.  Property \eqref{e:small-boundaries} in Theorem \ref{t:cubes} allows us to make our initial contraction as we did for $B(x,1)$ with $B(x,1-\epsilon)$.  The rest of the proof is exactly the same.
\end{remark}


\begin{bibdiv}
\begin{biblist}


\bib{breuillard}{misc}{
  title = {Geometry of locally compact groups of polynomial growth and shape of large balls},
  author = {Breuillard, E.},
  note = {{\tt arXiv:0704.0095}},
  year = {2007},
}

\bib{christ}{article}{
  title = {A $T(b)$ theorem with remarks on analytic capacity and the Cauchy integral},
  author = {Christ, M.},
  journal = {Colloq. Math.},
  volume = {60/61},
  number = {2},
  pages = {601-628},
  year = {1990},
}

\bib{david}{article}{
  title = {Morceaux de graphes Lipschitziens et int\'egrales singuli\`eres sur un surface},
  author = {David, G.},
  journal = {Rev. Mat. Iberoam.},
  volume = {4},
  number = {1},
  pages = {73-114},
  year = {1988},
}

\bib{david-wavelets}{book}{
  title = {Wavelets and singular integrals on curves and surfaces},
  author = {David, G.},
  series = {Lecture Notes in Mathematics},
  volume = {1465},
  publisher = {Springer-Verlag},
  year = {1991},
}

\bib{david-gc}{article}{
  title = {Bi-Lipschitz pieces between manifolds},
  author = {David, G.C.},
  journal = {Rev. Mat. Iberoam.},
  note = {To appear},
}

\bib{david-semmes}{article}{
  title = {Quantitative rectifiability and Lipschitz mappings},
  author = {David, G.},
  author = {Semmes, S.},
  journal = {Trans. Amer. Math. Soc.},
  volume = {337},
  number = {2},
  year = {1993},
  pages = {855-889},
}


\bib{guivarch}{article}{
  title = {Croissance polyn\^{o}miale et p\'eriodes des fonctions harmoniques},
  author = {Guivarc'h, Y.},
  journal = {Bull. Sc. Math. France},
  volume = {101},
  pages = {353-379},
  year = {1973},
}


\bib{heinonen-semmes}{article}{
  title = {Thirty-three yes or no questions about mappings, measures, and metrics},
  author = {Heinonen, J.},
  author = {Semmes, S.},
  journal = {Conform. Geom. Dyn.},
  volume = {1},
  year = {1997},
  pages = {1-12},
}

\bib{jones}{article}{
  title = {Lipschitz and bi-Lipschitz functions},
  author = {Jones, P.},
  journal = {Rev. Mat. Iberoam.},
  volume = {4},
  number = {1},
  year = {1988},
  pages = {115-121},
}

\bib{ledonne}{article}{
  title = {A metric characterization of Carnot groups},
  author = {Le Donne, E.},
  journal = {Proc. Amer. Math. Soc.},
  volume = {132},
  year = {2015},
  pages = {845-849},
}

\bib{llr}{misc}{
  title = {Ahlfors-regular distances on the Heisenberg group without biLipschitz pieces},
  author = {Le Donne, E.},
  author = {Li, S.},
  author = {Rajala, T.},
  year = {2015},
  note = {Preprint},
}

\bib{li}{article}{
  title = {Coarse differentiation and quantitative nonembeddability for Carnot groups},
  author = {Li, S.},
  journal = {J. Funct. Anal.},
  volume = {266},
  pages = {4616-4704},
  year = {2014},
}

\bib{magnani}{article}{
  title = {Differentiability and area formula on stratified Lie groups},
  author = {Magnani, V.},
  journal = {Houston J. Math.},
  volume = {27},
  number = {2},
  pages = {297-323},
  year = {2001},
}

\bib{meyerson}{article}{
  title = {Lipschitz and bilipschitz maps on Carnot groups},
  author = {Meyerson, W.},
  journal = {Pac. J. Math},
  volume = {263},
  number = {1},
  pages = {143-170},
  year = {2013},
}

\bib{montgomery}{book}{
  title = {A tour of sub-Riemannian geometries, their geodesics and applications},
  author = {Montgomery, R.},
  volume = {91},
  series = {Mathematical Surveys and Monographs},
  publisher = {American Mathematical Society},
  year = {2002},
}

\bib{pansu}{article}{
  title = {M\'etriques de Carnot-Carath\'eodory et quasiisom\'etries des espaces sym\'etriques de rang un},
  author = {Pansu, P.},
  journal = {Ann. Math. (2)},
  volume = {129},
  number = {1},
  pages = {1-60},
  year = {1989},
}


\bib{schul-AR}{article}{
  title = {Ahlfors-regular curves in metric spaces},
  author = {Schul, R.},
  journal = {Ann. Acad. Sci. Fenn. Math.},
  volume = {32},
  pages = {437-460},
  year = {2007},
}

\bib{schul}{article}{
  title = {Bi-Lipschitz decomposition of Lipschitz functions into a metric space},
  author = {Schul, R.},
  journal = {Rev. Mat. Iberoam.},
  volume = {25},
  number = {2},
  pages = {521-531},
  year = {2009},
}

\end{biblist}
\end{bibdiv}
\end{document}